\documentclass{amsart}
\usepackage{mathrsfs}
\usepackage{amssymb,latexsym,amsmath,amsthm,enumitem,mathrsfs}
\usepackage{hyperref}
\usepackage{color}
\usepackage{stmaryrd}
\usepackage{mathrsfs}
\usepackage{lineno}
\usepackage{bbm}

\usepackage{pgfplots}
\usepackage{graphicx}
\usetikzlibrary{math}

\usepackage{scalerel,stackengine}
\usepackage{tikz}

\usepackage{ esint }
\usepackage{graphicx}
\usepackage{amssymb}
\newtheorem{theorem}{Theorem}
\newtheorem{lemma}[theorem]{Lemma}
\newtheorem{remark}[theorem]{Remark}
\newtheorem{proposition}[theorem]{Proposition}

\newtheorem{definition}[theorem]{Definition}

\newtheorem{example}[theorem]{Example}
\newtheorem{conjecture}{Conjecture}

\theoremstyle{definition}

\newcommand{\cL}{\mathcal{L}}

\newcommand{\cV}{\mathcal{V}}
\newcommand{\cW}{\mathcal{W}}

\newcommand{\RR}{\mathbb R}

\newcommand{\F}{\mathbb F}

\newcommand{\e}{\epsilon}

\newcommand{\Tr}{\textup{Tr}}

\newcommand{\Fq}{\mathbb{F}_q}

\def\set4{\mathcal I}
\def\tup14{(1,2,3,4)}

\def\bv{{\mathbf v}}

\newtheorem*{comm*}{Comment}
\newtheorem*{lemma*}{Lemma}

\newcommand{\supp}{\mathrm{supp}}

\newcommand{\R}{\mathbb{R}}

\newcommand{\wh}{\widehat}

\newcommand{\si}{\sigma}

\newcommand{\Id}{\boldsymbol 1}
\newcommand{\Fp}{\mathbb{F}_p}
\newcommand{\bdA}{\mathbf{A}}
\newcommand{\bdE}{\mathbf{E}}

\begin{document}




 \author{Paige Bright}
 \address{Department of Mathematics\\
 Massachusetts Institute of Technology\\
 Cambridge, MA 02142-4307, USA}
 \email{paigeb@mit.edu}
 
 \author{Shengwen Gan}
 \address{Department of Mathematics\\
 Massachusetts Institute of Technology\\
 Cambridge, MA 02142-4307, USA}
 \email{shengwen@mit.edu}


\date{}

\title{Exceptional set estimates in finite fields}
\maketitle

\begin{abstract}
We study the exceptional set estimate for projections in $\F_q^n$. For each $V\in G(k,\mathbb{F}^n_q)$, let
$$ \pi_V: \mathbb{F}_q^n\rightarrow V $$
be the projection map.
We prove the following result:

If $A\subset \mathbb{F}_q^n$ with $\#A=q^a$ ($n-1\le a\le n$) and $0< s<\frac{a+n-2}{2}$, then 
\begin{equation*}
    \# \{V\in G(n-1,\mathbb{F}^n_q): \#\pi_V(A)< q^s  \}\lessapprox  q^{n-2}.
\end{equation*}
This improves the previous range $0<s<\frac{n-1}{n}a$.
Also, our range of $s$ is sharp in the sense that if $s>\frac{a+n-2}{2}$, then the right hand side above should be at least $q^t$ for some $t>n-2$.

\end{abstract}


\section{Introduction}

We first introduce the problem of exceptional set estimate in the Euclidean space. 
Let $G(k,\R^n)$ be the set of $k$-dimensional subspaces in $\R^n$ and $A(k,\R^n)$ be the set of $k$-dimensional affine subspaces in $\R^n$. For $W\in G(k,\R^n)$, define $\pi_W:\RR^n\rightarrow W$ to be the orthogonal projection onto $W$. Marstrand's projection theorem then states that the projection maps $\pi_W$ preserve Hausdorff dimension of Borel sets for almost every $W\in G(k,\R^n)$. More precisely, given $A\subset \R^n$, then
\[ \dim(\pi_W(A))=\min\{k,\dim(A)\} \]
for almost every $W\in G(k,\R^n)$.

The problem of exceptional set estimates is to find the set of $W$ such that the equation above fails. We fix a parameter $0<s<\min\{k,\dim(A)\}$, and define the exceptional set
\begin{equation}\label{exset}
    E_{s}(A):=\{W\in G(k,\R^n): \dim(\pi_W(A))<s\}.
\end{equation}

\begin{remark}
    \rm{We remark that $E_s(A)$ also depends on $k$, but we omit it as $k$ is clear from the context.}
\end{remark}

As far as the authors know, there are three types of the exceptional set estimates for the orthogonal projections $\{\pi_W: W\in G(k,\R^n)\}$. We state these results. For simplicity, we denote $\dim(A)=a$.

\begin{enumerate}[label=(\roman*)]
    \item $\dim (E_s(A))\le k(n-k)+s-k$. \label{1}
    \item $\dim (E_s(A))\le \max\{k(n-k)+s-a,0\}$.\label{2}
    \item $\dim(E_{\frac{k}{n}a}(A))\le k(n-k)-1$.\label{3}
\end{enumerate}

The first one is known as the Kaufman-type estimate (\cite{kaufman1968hausdorff}).
The second one is known as the Falconer-type estimate (\cite{falconer1982hausdorff}). The Falconer-type estimate was also proved by Peres and Schlag \cite{peres2000smoothness}.
The third one is due to He (\cite{he2020orthogonal}).

\bigskip

\subsection{Motivation}
We first consider a special case when $n=3,k=2$ and $\dim A=2$. Both \ref{1} and \ref{2} imply that for any $A\subset \R^3$ with $\dim A=2$, 
\begin{equation}\label{ex0}
    \dim(E_s(A))=\dim(\{V\in G(2,\R^3): \dim(\pi_V(A)\})<s\})\le s. \
\end{equation} 
This estimate is sharp when $s\rightarrow 1^+$. The example is when $A=\R^2\times \{0\}$. For such $A$ and $1<s<2$, we see that $E_s(A)=\{V\in G(2,\R^2): V\textup{~parallel~to~}(0,0,1)\}$, which has dimension $1$. When $s$ is essentially bigger than $1$, the example does not match the upper bound in \eqref{ex0}, so we suspect \eqref{ex0} is not sharp when $s>1$. In fact, \ref{3} shows that
\[ \dim(E_{\frac{4}{3}}(A))\le 1. \]

We begin to think about the following question. Suppose we consider the case $n=3,k=2$. What is the largest $s$, so that for any $\dim(A)=2$,
\[ \dim(E_s(A))\le 1? \]

In this paper, we will show in the finite field setting that this largest $s$ is $\frac{3}{2}$. We will prove a general theorem (Theorem \ref{mainthm}) for all the dimensions. We will also discuss the obstruction to generalize the finite field version to $\R^n$ at the end of the paper.

\subsection{Finite field version}
Let $p$ be a prime number and $q=p^r$.
The goal of this paper is to study exceptional set estimates over the finite field $\F_q^n$.
For any subspace $V\subset \Fq^n$, we can define the orthogonal projection
\[ \pi_V:\Fq^n\rightarrow V \]
in the natural way. (For precise statement, see Definition \ref{defpiv}.)

\begin{definition}[Exceptional set in finite field]
For $A\subset \Fq^n$ and a number $s>0$, we define the $s$-exceptional set of $A$ for projection to $k$-planes  to be
\begin{equation}\label{defex}
    E_s(A):=\{ V\in G(k,\Fq^n): \#\pi_V(A)<q^s \}.
\end{equation}
\end{definition}
\begin{remark}
\rm
One should think of $E_s(A)=E_s^{n,k}(A)$ depends on another two parameters: $n$ the ambient dimension; $k$ the dimension of the planes where $A$ is projected to. 
Since $k,n$ will be clear, we just drop $n,k$ from our notation for simplicity.
\end{remark}

It is not surprising that \ref{1}, \ref{2} and \ref{3} have their corresponding finite field version. The following result is obtained by Chen \cite[Theorem 1.2]{chen2018projections}.

\begin{proposition}[Chen]\label{prop6}
    Let $A\subset \Fq^n$ be a set with $\#A=q^a$ $(0<a<n)$. For $s\in (0,\min\{k,a\})$, we have the following estimates:
\begin{enumerate}[label=(\roman*)]
    \item (Finite field version of \textup{\ref{1}}) $\# E_s(A)\lesssim q^{k(n-k)+s-k}$;\label{01}
    \item (Finite field version of \textup{\ref{2}}) $\# E_s(A)\lesssim q^{\max\{k(n-k)+s-a,0\}}$.\label{02}
\end{enumerate}
\end{proposition}
\begin{remark}
{\rm
    The original result is stated for the prime field $\Fp$, but we believe it also holds for the general finite field $\Fq$. There may also be the finite field version of \ref{3}, but unfortunately we did not find the reference. Though, we still state the finite field version of \ref{3} to compare with our theorem. }    
\end{remark}

\noindent 
\textit{(iii) (Finite field version of} \ref{3}\textit{)} $\# E_s(A)\lesssim q^{k(n-k)-1}$, for $s<\frac{k}{n}a$.

\bigskip

Next, we state our main theorem.
\begin{theorem}\label{mainthm}
    Let $A\subset \Fq^n$ be a set with $\#A=p^a$ $(0<a<n)$.
Then for $0<s<\frac{a+2k-n}{2}$, we have
\begin{equation}\label{mainthmest}
    \# \{ V\in G(k,\Fq^n): \#\pi_V(A)<q^s \}\le C_{n,k,a,s}\cdot\log q\cdot q^{t(a,s)},
\end{equation}
where $t(a,s)=\max\{k(n-k)+2(s-a), (k-1)(n-k)\}$. Here, $C_{n,k,a,s}$ is a constant that may depend on $n,k,a,s$, but not depend on $p$.
\end{theorem}

In the theorem, one particularly interesting case is when $k=n-1$, $a\ge n-1$ and $s>a-1$. This range of parameters is our motivation to work on this project, so we would like to also state the theorem for this specific range of parameters. 

\begin{theorem}\label{spethm}
    Let $A\subset \Fq^n$ be a set with $\#A=p^a$ $(n-1\le a<n)$. 
Then for $0< s<\frac{a+n-2}{2}$, we have
\begin{equation}\label{spethmest}
    \# \{ V\in G(n-1,\Fq^n): \#\pi_V(A)<q^s \}\lesssim C_{n,a,s}\cdot\log q\cdot q^{n-2}.
\end{equation}
\end{theorem}

\begin{figure}
\begin{tikzpicture}
\begin{centering}
\tikzmath{
\r1=3;
\r2=5;
\q1 = 5;
\q2 = 10;
\q3 = 15;
\q4 = 20;
\x1 = \q1; \y1 = 5;
\x2 = \q2; \y2 = 5;
\x3 = \q3; \y3 = 5;
\x4 = \q4; \y4 = 8;
\x5 = \q4; \y5 = 10;
} 
\begin{axis}[
axis x line=middle,
axis y line=middle,
xtick={\x1,\x2,\x3,\x4,\x4+5},
xticklabels={$a-1$,$a \cdot \frac{n-1}{n}$,$\frac{a+n-2}{2}$,$n-1$,$\mathbf{s}$},
xlabel near ticks,
ytick={},
yticklabels={},
xlabel near ticks,
xmax=\x4+5,
ymax=\y5+4,
xmin=0,
ymin=0,
ylabel=$\mathbf{t(a,s)}$,
]
\addplot[domain=\q1:\q2] {5};
\addplot[domain=\q2:\q3] {5};
\addplot[domain=\q3:\q4] {x-10};
\draw[dotted] (axis cs:\x1,\y1) -- (axis cs:\x1, 0);
\draw[dotted] (axis cs:\x2,\y2) -- (axis cs:\x2, 0);
\draw[dotted] (axis cs:\x3,\y3) -- (axis cs:\x3, 0);
\draw[dotted] (axis cs:\x4,\y5) -- (axis cs:\x4, 0);
\addplot[only marks,mark=*] coordinates{(5,5)(10,5)(15,5)(20,10)};
\node [above] at (49,53) {$A$};
\node [above] at (98,53) {$B$};
\node [above] at (147,53) {$C$};
\node [above] at (210,100) {$D$};

\end{axis}
\end{centering}
\end{tikzpicture}

\caption{The range of $t$ when $k=n-1$, $a\ge n-1$ and $s>a-1$}
\label{grph}
\end{figure}
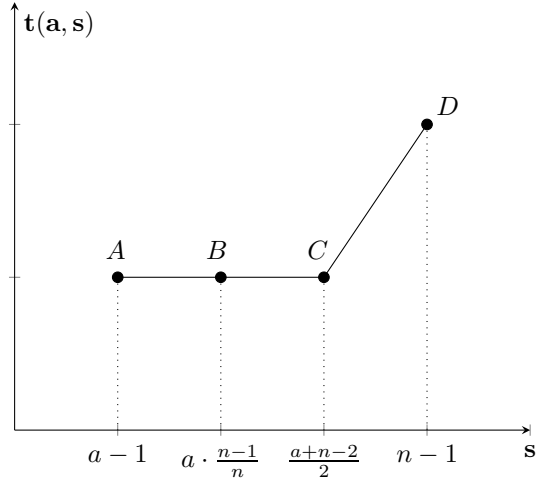

\begin{remark}
\rm
We will discuss how sharp our estimate \eqref{spethmest} is. We look at Figure \ref{grph}: The three points $A,B,C$ have the same ordinate $n-2$ and respective abscissas indicated in the figure; The point $D$ has the coordinate $(n-1,2(n-1)-a)$; The line segment $C-D$ is given by the equation $t(a,s)=2s-a$. We will construct examples in Section \ref{section3} to show that the necessary condition for \eqref{mainthmest} to hold is when $(s,t(a,s))$ lies above the graph $A-B-C-D$. The estimate \ref{3} indicates that a sufficient condition for \eqref{mainthmest} to hold is when $(s,t(a,s))$ lies above $A-B$. Our Theorem \ref{spethm} extends \ref{3} by showing that $\eqref{mainthmest}$ holds when $(s,t(a,s))$ lies above $A-B-C$. We see that our theorem finds the optimal range of $s$ for which $t(a,s)$ could be $n-2$.

\medskip

Though we proved the estimate in finite fields, it may be reasonable to ask whether it is able to prove the analogue in $\R^n$: For $\dim(A)=a$ and $s<\frac{a+2k-n}{2}$, do we have
\begin{equation}
    \dim(E_s(A))\le \max\{k(n-k)+2(s-a), (k-1)(n-k)\}? 
\end{equation} 
However, it is hard to generalize our proof to the Euclidean space. We will talk about the obstruction at the end of the paper.

\end{remark}


We talk about the structure of the paper. In Section \ref{section3}, we talk about some examples. In Section \ref{section4}, we briefly review the Fourier transform in finite fields. In Section \ref{section5}, we prove Theorem \ref{mainthm}. 


\bigskip

\begin{sloppypar}
\noindent {\bf Acknowledgement.}
We would like to thank Prof. Larry Guth for numerous helpful discussions over the course of this project.
\end{sloppypar}



\section{Definition of the orthogonal projections}\label{section2}

Let $p$ be a prime number and $q=p^r$.
The goal of this paper is to study exceptional set estimates over the finite field $\F_q^n$. We first note that $\F_q^n$ is equipped with a natural non-degenerate bilinear form $\F_q^n\times \F_q^n\rightarrow \F_q$ given by
\[ x\cdot y:=x_1y_1+\dots+x_ny_n\ \ \ \textup{for~}x, y\in\F_q^n. \]
We remark that this bilinear form is not necessarily an inner product. For example, if $n=p$ and $x=(1,\dots,1)\in\F_q^p$, then $x\cdot x=0$. However, we can still use this bilinear form to define ``orthogonality", even though some non-zero vector may be orthogonal to itself.

\begin{definition}
    We say two vectors $x,y\in\F_q^n$ are orthogonal, denoted by $x\perp y$, if $x\cdot y=0$. For $V\subset \F_q^n$ being a $k$-dimensional subspace whose directions are spanned by the vectors $v_1,\dots,v_k$, we say $x$ is orthogonal to $V$, denoted by $x\perp V$, if $x\perp v_1,\dots,v_k$. We also define the orthogonal complement of $V$ to be
    \[ V^\perp:=\{x\in\F_q^n: x\perp V\}. \]
\end{definition}
\begin{remark}
    {\rm
    By linear algebra, we see that $V^\perp$ is an $(n-k)$-dimensional subspace. Later we will see that $V^\perp$ is exactly the dual of $V$ which is defined in Definition \ref{defdual}.
    }
\end{remark} 

Since in $\F_q^n$ some vector may be orthogonal to itself, we need a slightly trickier definition of the orthogonal projection. Recall that for $V\in G(k,\R^n)$, $\pi_V: \R^n\rightarrow V$ is the orthogonal projection onto $V$. We can also identify $\pi_V$ as the map
\begin{equation}
     \R^n\rightarrow A(n-k,\R^n),
\end{equation} 
so that $\pi_V(x)$ is the unique $(n-k)$-dimensional plane that is parallel to $V^\perp$ and passes through $x$. This motivates the definition of projection in finite fields.

\begin{definition}\label{defpiv}
Let $\F_q$ be a finite field. Denote the $k$-dimensional subspaces and $k$-dimensional affine subspaces of $\Fq^n$ by $G(k,\F^n_q)$ and $A(k,\Fq^n)$, respectively. For $V\in G(k,\Fq^n)$, define
\begin{equation}\label{defpiveq}
    \pi_V: \Fq^n\rightarrow A(n-k,\F_q^n),
\end{equation} 
so that $\pi_V(x)$ is the unique element in $A(n-k,\F_q^n)$ that is parallel to $V^\perp$ and passes through $x$. 
\end{definition}

\section{Examples of the exceptional sets in the prime field}\label{section3}

We discuss some examples of the exceptional sets. In $\R^2$, it is conjectured that (see Section 5.4 in \cite{mattila2015fourier}):

\begin{conjecture}\label{conj1}
Let $A\subset \R^2$ with $\dim(A)=a$. For $\theta\in G(1,\R^2)$, let $\pi_\theta: \R^2\rightarrow \theta$ be the orthogonal projection onto line $\theta$. For $0<s<\max\{1,a\}$, define $E_s(A):=\{\theta: \dim(\pi_\theta(A))<s\}$. Then
\[\dim(E_s(A))\le \max\{0,2s-a\}.\]
\end{conjecture}

Of course, we can also ask the question in $\Fp^2$:
\begin{conjecture}\label{conj1.2}
Let $A\subset \Fp^2$ with $\#A=p^a$. For $0<s<\max\{1,a\}$, define $ E_s(A)$ as in \eqref{defex} with $n=2,k=1$. Then
\begin{equation}\label{conj2ineq}
    \#E_s(A) \le C_{\e,a,s} p^{\e+\max\{0,2s-a\}}.
\end{equation}
\end{conjecture}

\begin{remark}
   \rm \eqref{conj2ineq} may not be true for $\F_q$ because of the existence of the subfield. The existence of the subfield causes the Szemer\'edi-Trotter theorem to fail in $\Fq^2$ (see the Introduction of \cite{mohammadi2018szemer}). However, it is reasonable to believe \eqref{conj2ineq} is true over the prime field.
\end{remark}

We show that the upper bound in \eqref{conj2ineq} can be attained.

\begin{example}\label{ex1}
{\rm
Let $p$ be a large prime. We assume $\frac{a}{2}<s<\max\{1,a\}$, $a<2$.  Consider a set of lines
$\cL=\{l_{k,m}: |k|\le p^{2s-a}, |m|\le 10 p^s\}$ in $\Fp^2$. Here, $k$ and $m$ are integers and $l_{k,m}$ is given by
\[ l_{k.m}: y=kx+m. \]
We used the convention that if an integer $m$ satisfies $|m|<p/2$, then $m$ can be naturally viewed as an element in $\Fp$. Hence, the $l_{k,m}$ defined above is a line in $\Fp^2$. 

We see that $\cL$ consists of lines from $\sim p^{2s-a}$ many directions, and in each of these directions there are $\sim p^s$ many lines. We denote these directions by $E$, and for $\theta\in E$, let $\cL_\theta$ be the lines in $\cL$ that are in direction $\theta$.

Consider the set $A:=\{(x,y)\in \Fp^2: |x|\le p^{a-s}, |y|\le p^s\}$. For any $|k|\le p^{2s-a}$, we see that every $(x,y)\in A$ satisfies $|y-kx|\le 10p^s$. This means that for any direction $\theta\in E$, $A$ is covered by $\cL_\theta$. Therefore we have for each $\theta\in E$,
\[ \#\pi_\theta(A)\le \#\cL_\theta\lesssim p^s. \] We obtain the following estimate:
\[ \#\{\theta: \#\pi_\theta(A)\lesssim p^s\}\ge \#E\sim p^{2s-a}. \]

Therefore, we showed that under the setting of Conjecture \ref{conj1.2}, there exists $A$ such that 
\[ \# E_s(A)\gtrsim p^{\max\{0,2s-a\}}. \]

}

\end{example}

From now on, we denote such sets $A$ and $E$ by $\bdA_{s,a}(\subset \Fp^2)$ and $\bdE_{s,a}(\subset G(1,\Fp^2))$. They serve as a tight example for \eqref{conj2ineq}. Later, we will use $\bdA_{s,a}$ and $\bdE_{s,a}$ as building blocks to build more examples in higher dimensions.

\medskip

For simplicity, we define $t(a,s)$ to be the smallest number such that for any $A\subset \Fp^n$ with $\#A=p^a$,
\[ \#E_s(A)\lesssim_{n,k,a,s,\e}p^{\e+t(a,s)}, \]
for any $\e>0$.
Equivalently,
\[ t(a,s)=\overline{\lim_{p\rightarrow\infty}}\sup_{A\subset \Fp^n, \#A=p^a}\log_p(\#E_s(A)). \]

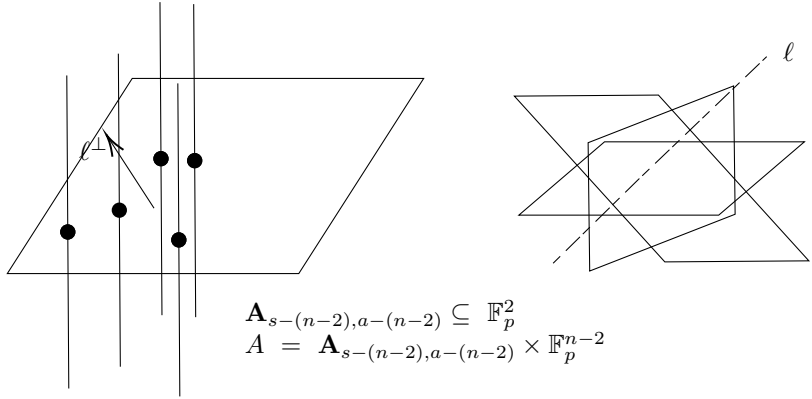
\begin{figure}
\begin{tikzpicture}[x=0.75pt,y=0.75pt,yscale=-1,xscale=1]

\draw  [color={rgb, 255:red, 0; green, 0; blue, 0 }  ,draw opacity=1 ] (163,71.5) -- (310,71.5) -- (247,170) -- (100,170) -- cycle ;
\draw [color={rgb, 255:red, 0; green, 0; blue, 0 }  ,draw opacity=1 ]   (130,70) -- (131,228) ;
\draw [color={rgb, 255:red, 0; green, 0; blue, 0 }  ,draw opacity=1 ]   (156,59) -- (157,217) ;
\draw [color={rgb, 255:red, 0; green, 0; blue, 0 }  ,draw opacity=1 ]   (177,33) -- (178,191) ;
\draw [color={rgb, 255:red, 0; green, 0; blue, 0 }  ,draw opacity=1 ]   (194,34) -- (195,192) ;
\draw [color={rgb, 255:red, 0; green, 0; blue, 0 }  ,draw opacity=1 ]   (186,74) -- (187,232) ;
\draw [color={rgb, 255:red, 0; green, 0; blue, 0 }  ,draw opacity=1 ] (174,137) -- (150,100);


\draw [shift={(150,100)}, rotate = 58] [color={rgb, 255:red, 0; green, 0; blue, 0 }  ,draw opacity=1 ][line width=0.75]    (10.93,-3.29) .. controls (6.95,-1.4) and (3.31,-0.3) .. (0,0) .. controls (3.31,0.3) and (6.95,1.4) .. (10.93,3.29)   ;

\draw [color={rgb, 255:red, 0; green, 0; blue, 0 }  ,draw opacity=1 ]   (130.5,149) ;
\draw [shift={(130.5,149)}, rotate = 0] [color={rgb, 255:red, 0; green, 0; blue, 0 }  ,draw opacity=1 ][fill={rgb, 255:red, 0; green, 0; blue, 0 }  ,fill opacity=1 ][line width=0.75]      (0, 0) circle [x radius= 3.35, y radius= 3.35]   ;
\draw [shift={(130.5,149)}, rotate = 0] [color={rgb, 255:red, 0; green, 0; blue, 0 }  ,draw opacity=1 ][fill={rgb, 255:red, 0; green, 0; blue, 0 }  ,fill opacity=1 ][line width=0.75]      (0, 0) circle [x radius= 3.35, y radius= 3.35]   ;
\draw [color={rgb, 255:red, 0; green, 0; blue, 0 }  ,draw opacity=1 ]   (156.5,138) ;
\draw [shift={(156.5,138)}, rotate = 0] [color={rgb, 255:red, 0; green, 0; blue, 0 }  ,draw opacity=1 ][fill={rgb, 255:red, 0; green, 0; blue, 0 }  ,fill opacity=1 ][line width=0.75]      (0, 0) circle [x radius= 3.35, y radius= 3.35]   ;
 
\draw [color={rgb, 255:red, 0; green, 0; blue, 0 }  ,draw opacity=1 ]   (177.5,112) ;
\draw [shift={(177.5,112)}, rotate = 0] [color={rgb, 255:red, 0; green, 0; blue, 0 }  ,draw opacity=1 ][fill={rgb, 255:red, 0; green, 0; blue, 0 }  ,fill opacity=1 ][line width=0.75]      (0, 0) circle [x radius= 3.35, y radius= 3.35]   ;
\draw [shift={(177.5,112)}, rotate = 0] [color={rgb, 255:red, 0; green, 0; blue, 0 }  ,draw opacity=1 ][fill={rgb, 255:red, 0; green, 0; blue, 0 }  ,fill opacity=1 ][line width=0.75]      (0, 0) circle [x radius= 3.35, y radius= 3.35]   ;
\draw [color={rgb, 255:red, 0; green, 0; blue, 0 }  ,draw opacity=1 ]   (194.5,113) ;
\draw [shift={(194.5,113)}, rotate = 0] [color={rgb, 255:red, 0; green, 0; blue, 0 }  ,draw opacity=1 ][fill={rgb, 255:red, 0; green, 0; blue, 0 }  ,fill opacity=1 ][line width=0.75]      (0, 0) circle [x radius= 3.35, y radius= 3.35]   ;
\draw [shift={(194.5,113)}, rotate = 0] [color={rgb, 255:red, 0; green, 0; blue, 0 }  ,draw opacity=1 ][fill={rgb, 255:red, 0; green, 0; blue, 0 }  ,fill opacity=1 ][line width=0.75]      (0, 0) circle [x radius= 3.35, y radius= 3.35]   ;
\draw [color={rgb, 255:red, 0; green, 0; blue, 0 }  ,draw opacity=1 ]   (186.5,153) ;
\draw [shift={(186.5,153)}, rotate = 0] [color={rgb, 255:red, 0; green, 0; blue, 0 }  ,draw opacity=1 ][fill={rgb, 255:red, 0; green, 0; blue, 0 }  ,fill opacity=1 ][line width=0.75]      (0, 0) circle [x radius= 3.35, y radius= 3.35]   ;
\draw [shift={(186.5,153)}, rotate = 0] [color={rgb, 255:red, 0; green, 0; blue, 0 }  ,draw opacity=1 ][fill={rgb, 255:red, 0; green, 0; blue, 0 }  ,fill opacity=1 ][line width=0.75]      (0, 0) circle [x radius= 3.35, y radius= 3.35]   ;
\draw   (392.93,103.95) -- (393.72,168.71) -- (467.07,140.05) -- (466.28,75.29) -- cycle ;
\draw   (458.8,140.5) -- (358,140.5) -- (401.2,103.5) -- (502,103.5) -- cycle ;
\draw   (431.61,164.1) -- (355.74,80.43) -- (428.39,79.9) -- (504.26,163.57) -- cycle ;
\draw  [dash pattern={on 3.75pt off 3pt on 7.5pt off 1.5pt}]  (375.28,164.79) -- (418.12,122.66) -- (483,60.5) ;

\draw (212,181.4) node [anchor=north west][inner sep=0.75pt]    {$ \begin{array}{l}
\bdA_{s-(n-2),a-(n-2)}\subseteq \ \Fp^{2}\\
A\ =\ \bdA_{s-(n-2),a-(n-2)}\times \Fp^{n-2}
\end{array}$};
\draw (135,100) node [anchor=north west][inner sep=0.75pt]    {$\ell^\perp $};
\draw (490,52.4) node [anchor=north west][inner sep=0.75pt]    {$\ell$};
\end{tikzpicture}

\caption{Projection to planes}
\label{projectiontoplane}
\end{figure}

\subsection{Sharpness of Theorem \ref{spethm}}
In this subsection, we discuss how sharp our Theorem \ref{mainthm} is.

Consider the case when $k=n-1, a\ge n-1$. Theorem \ref{spethm} gives the upper bound:
\[   t(a,s)\le n-2,\ \textup{when~}s<\frac{a+n-2}{2}. \]
On the other hand, we can construct examples to show that
\[ t(a,s)\ge n-2,\ \textup{when~}a\ge n-1,\ a-1<s<n-1. \]

Let $A=\Fp^{n-1}\times I$ where $I\subset \Fp$ with $\# I=\lfloor p^{a-(n-1)}\rfloor$. We want to see what is $E_s(A)$. For $V\in G(n-1,\Fp^n)$, if $(0,\dots,0,1)\notin V$, then $\pi_V$ restricts to an isomorphism on $\Fp^{n-1}\times \{0\}$ \[\pi_V: \Fp^{n-1}\times \{0\}\xrightarrow{\sim} V.\]
Therefore, $\#\pi_V(A)=\#\Fp^{n-1}=p^{n-1}>p^s$ which implies $V\notin E_s(A)$. If $(0,\dots,0,1)\in V$, then $\pi_V(A)=\F_p^{n-2}\times I$. Therefore, $\#\pi_V(A)=\#\F_p^{n-2}\#I\le p^{a-1}<p^s$, which implies $V\in E_s(A)$. We have that $E_s(A)= \{V\in G(n-1,\Fp^n): (0,\dots,0,1)\in V \}\cong G(n-2,\Fp^{n-1})$, which implies 
\[ \#E_s(A)\sim p^{n-2}. \]
This shows that when $s\in(a-1,\frac{a+n-2}{2})$, in order for \eqref{mainthmest} to hold, $(s,t(a,s))$ must lie above the graph $A-B-C$ in Figure \ref{grph}.

Also, our estimate is sharp in another sense: the range of $s$ is sharp. Actually, we show that if $s>\frac{a+n-2}{2}$, then
\[ t(a,s)\ge 2s-a(>n-2). \]
In other words, in order for \eqref{mainthmest} to hold, $(s,t(a,s))$ must lie above the graph $C-D$ in Figure \ref{grph}.

Choose $A=\bdA_{s-(n-2),a-(n-2)}\times \Fp^{n-2}$ (see Figure \ref{projectiontoplane}). We first look at those $(n-1)$-subspaces $V$ that contains $\{(0,0)\}\times\Fp^{n-2}$. Such $V$ can be written as $\ell\times \Fp^{n-2}$, where $\ell\in G(1,\Fp^2)$. 
It is not hard to see that $\pi_V(A)=\pi_{\ell}(\bdA_{s-(n-2),a-(n-2)})\times \Fp^{n-2}$ (where we view $\pi_\ell:\Fp^2\rightarrow \ell$).
Therefore, we obtain that
\[ \#\pi_V(A)=\#\Fp^{n-2} \#\pi_\ell(\bdA_{s-(n-2),a-(n-2)}). \]
If $\ell\in \bdE_{s-(n-2),a-(n-2)}$, then $\#\pi_\ell(\bdA_{s-(n-2),a-(n-2)})<p^{s-(n-2)}$, and hence
\[ \#\pi_V(A)< p^{s}.  \]
To indicate the relation between $V$ and $\ell$, we denote $V_\ell=\ell\times \Fp^{n-2}$. We have shown that if $\ell\in \bdE_{s-(n-2),a-(n-2)}$ then 
\[ V_\ell \in E_s(A). \]

Let $\ell^\perp\subset \Fp^2$ be the line orthogonal to $\ell$. By abuse of notation, we also use $\ell^\perp$ to denote the vector in $\Fp^2$.
Consider another $(n-1)$-subspace $W$ whose normal direction is of form $\ell^\perp+y$ for $y\in \{(0,0)\}\times \Fp^{n-2}$ (see right hand side of Figure \ref{projectiontoplane} for such $W$'s). We remark that $\ell^\perp$ is the normal direction of $V_\ell$. We claim that $\#\pi_{V_\ell}(A)=\#\pi_W(A)$. Note that $\#\pi_{V_l}(A)$ is the number of lines parallel to $\ell^\perp$ that are needed to cover $A$. It is equal to the number of lines parallel to $\ell^\perp+y$ that are needed to cover $A$. Based on one $V_\ell$, we find another $\sim p^{n-2}$ many $W$'s that are in $E_s(A)$. Therefore,
\[ \#E_s(A)\gtrsim \#\bdE_{s-(n-2),a-(n-2)} p^{n-2}\gtrsim p^{2\big(s-(n-2)\big)-\big(a-(n-2)\big)+n-2}=p^{2s-a}. \]

\section{Fourier transform in finite field}\label{section4}

\subsection{Definition of Fourier transform}
We briefly introduce the Fourier transform in $\mathbb{F}_q^n$. We first set up our notation. $\Fq^n$ is our physical space, and we use $x,y$ to denote the points in $\Fq^n$. The frequency space is also $\Fq^n$, and we use $\xi,\eta$ to denote the points in it. For $x\in\Fq^n$ or $\xi\in\Fq^n$, we also write $x=(x_1,\dots,x_n)$ or $\xi=(\xi_1,\dots,\xi_n)$ in coordinate, where each $x_i$ or $\xi_i$ belongs to $\Fq$.

Before giving the definition of the Fourier transform, we need to introduce some notation from number theory.
Recall that $q=p^r$. Define the trace map
\[ \Tr: \Fq\rightarrow \Fp,\ \ x\mapsto \sum_{i=0}^{r-1} x^{p^i}. \]
First, we need to explain why $\Tr(x)\in \Fp$. Let 
\[\si: \Fq\rightarrow \Fq,\ \  x\mapsto x^p\] 
be the Frobenius map. Then we can write $\Tr=\sum_{i=0}^{r-1}\si^i$. By a fundamental fact in number theory, we know that the Galois group $\textup{Gal}(\F_q/\F_p)$ is generated by $\si$. Since $\si(\Tr(x))=\Tr(x)$, we see that $\Tr(x)$ is invariant under the Galois group, and hence $\Tr(x)\in\Fp$.
Another two important properties for $\Tr$ is that $\Tr$ is $\Fp$-linear, i.e., $\Tr(x+y)=\Tr(x)+\Tr(y)$ and $\Tr(\lambda x)=\lambda \Tr(x)$ for $\lambda\in\Fp$.
\bigskip

We are ready to define the Fourier transform on $\Fq^n$.
For a function $f(x)$ on $\mathbb{F}_q^n$, the Fourier transform of $f$ is a function on the frequency space $\Fq^n$ given by
\[
\wh{f} (\xi) = \sum_{x\in \mathbb{F}_q^n} f(x) e_p (-\Tr(x\cdot \xi)).
\]
Here, $e_p(x) = e^{\frac{2\pi i x}{p}}$, and we view $\Tr(x\cdot\xi)$ as an element in $\{1,2,\dots,p\}$. 
For a function $g(\xi)$ on $\Fq^n$, the inverse Fourier transform of $g$ is a function on $\Fq^n$ given by 
\[
g^{\vee}(x) = \frac{1}{q^n} \sum_{\xi \in \Fq^n} g(\xi) e_p(\Tr(x\cdot \xi)).
\] 

We will prove the Fourier inversion theorem and the Plancherel's identity.

\begin{lemma}\label{lemfin}
    For $x\in\Fq$, we have 
    \[ \sum_{\xi\in\Fq}e_p(\Tr(x\xi))=q \cdot 1_{x=0}. \]
\end{lemma}

\begin{proof}
    When $x=0$, then the left hand side equals $q$. When $x\neq 0$, the left hand side equals
\begin{align*}
\sum_{\xi\in\Fq}e_p(\Tr(\xi)).
    \end{align*}
We just need to show that for any $y\in \Fp$, the number of $\xi$ such that $\Tr(\xi)=y$ are all the same. Then we have    
\begin{align*}
\sum_{\xi\in\Fq}e_p(\Tr(\xi))=\frac{q}{p}\sum_{y\in\Fp}e_p(y)=0.
    \end{align*}
To calculate $\#\{\xi\in\Fq: \Tr(\xi)=y\}$, we first find a $\xi_0\in\Fq$ such that $\Tr(\xi_0)=1$. Note that
\[ \Tr(\xi)=\sum_{i=0}^{r-1}\xi^{p^i} \]
is a polynomial of degree $p^{r-1}$. Therefore, there exists $\xi_1\in\Fq$ such that $\Tr(\xi_1)\neq 0$. Choosing $\xi_0=\Tr(\xi_1)^{-1}\xi_1$ and noting that $\Tr(\xi_1)^{p^i}=\Tr(\xi_1)$ because of $\Tr(\xi_1)\in\Fp$, we have
\[ \Tr(\xi_0)=\sum_{i=0}^{r-1}(\Tr(\xi_1)^{-1}\xi_1)^{p^i}=\Tr(\xi_1)^{-1}\sum_{i=0}^{r-1}\xi_1^{p^i}=1. \]

Now we can see that for any $y,y'\in\Fp$, if  $\Tr(\xi)=y$, then by the $\F_p$-linearity of $\Tr$, $\Tr(\xi+(y-y')\xi_0)=y'$. This shows that $\#\{\xi\in\Fq: \Tr(\xi)=y\}$ are the same for all $y\in\Fp$.
\end{proof}

\begin{lemma}[Fourier inversion]
    If $f$ is a function on $\Fq^n$, then $(\wh f)^\vee=f$.
\end{lemma}
\begin{proof}
  By definition,
  \begin{align*}
      (\wh f)^\vee(x)&=\frac{1}{q^n} \sum_{\xi\in\Fq^n} \sum_{y\in\Fq^n} f(y)  e_p(-\Tr(y\cdot\xi)) e_p(\Tr(x\cdot\xi))\\
      &=\frac{1}{q^n} \sum_{y\in\Fq^n}\sum_{\xi\in\Fq^n}  f(y)  e_p\bigg(\Tr\Big((x-y)\cdot\xi\Big)\bigg)\\
      &=\frac{1}{q^n}\sum_{y\in\Fq^n}\sum_{\xi\in\Fq^n}  f(y)  \prod_{j=0}^n e_p\bigg(\Tr\Big((x_j-y_j)\xi_j\Big)\bigg)
  \end{align*} 
By Lemma \ref{lemfin}, it equals to
\[ \frac{1}{q^n}\sum_{y\in\Fq^n}f(y)\prod_{j=0}^n q 1_{y_j=x_j}=f(x).  \]
\end{proof}

\begin{lemma}[Plancherel's identity]
\[ \sum_{x\in\Fq^n}|f(x)|^2=\frac{1}{q^n}\sum_{\xi\in\Fq^n}|\wh f(\xi)|^2. \]
\end{lemma}
\begin{proof}
    \begin{align*}
        \frac{1}{q^n}\sum_{\xi\in \Fq^n}|\wh f(\xi)|^2&= \frac{1}{q^n} \sum_{\xi\in\Fq^n}\left| \sum_{x\in\Fq^n} f(x) e_p(-\Tr(x\cdot\xi)) \right|^2\\
        &=\frac{1}{q^n} \sum_{\xi\in\Fq^n} \sum_{x\in\Fq^n}\sum_{x'\in\Fq^n} f(x)\overline{ f(x')} e_p(\Tr((x'-x)\cdot\xi))\\
        &=\sum_{x\in\Fq^n}|f(x)|^2.
    \end{align*}
\end{proof}

\subsection{Dual space}

\begin{definition}[Dual space]\label{defdual}
    For $V\in G(k,\Fq^n)$, we define $V^*=\supp \wh \Id_V$. 
\end{definition}

The intuition in $\R^n$ is that: if $V\in G(k,\R^n)$, then $\supp\wh \Id_V=V^\perp$ (viewing $\Id_V$ as a distribution). Therefore, the dual space of $V$ is $V^\perp\in G(n-k,\R^n)$. We will show that for finite field, it is also true that $V^*=V^\perp$.

\begin{lemma}\label{lem1}

    If $V\in G(k,\Fq^n)$, then $V^*=V^\perp$. Moreover, $\wh \Id_V=q^k \Id_{V^*}$.
\end{lemma}

\begin{proof}
    Suppose the $k$-dimensional space $V$ is spanned by the following $k$ vectors: 
    \begin{equation*}
      \bv_1=(v_{11},\cdots,v_{1n}),
      \bv_2=(v_{21},\cdots,v_{2n}), \cdots
      \bv_k=(v_{k1},\cdots,v_{kn}).
\end{equation*}
We use $\cV$ to denote the $k\times n$ matrix 
\[ \cV=\begin{pmatrix}
    \bv_1\\
    \bv_2\\
    \vdots\\
    \bv_k
\end{pmatrix}. \]
Therefore $V$ can be written as
\begin{equation*}
    V=\{ (y_1,\dots,y_k)\cV : y_1,\dots,y_k\in\Fq \}.
\end{equation*}

We will calculate $\wh \Id_V$. By definition
\begin{equation}\label{tocalculate}
    \wh \Id_V= \sum_{x\in V} e_p(-\Tr(x\cdot\xi)) = \sum_{y_1,\dots,y_k\in\Fq} e_p(-\Tr\bigg((y_1,\dots,y_k)\cV \begin{pmatrix}
        \xi_1\\
        \vdots\\
        \xi_n
    \end{pmatrix}\bigg)). 
\end{equation}
To calculate the right hand side, we first choose $\bv_{k+1},\dots,\bv_{n}$, so that $\{\bv_1,\dots,\bv_n\}$ form a basis of $\Fq^n$. Define 
\[ \cW=\begin{pmatrix}
    \cV\\
    \bv_{k+1}\\
    \vdots\\
    \bv_n
\end{pmatrix}=\begin{pmatrix}
    \bv_1\\
    \vdots\\
    \bv_n
\end{pmatrix},\]
which is invertible. We can write the right hand side of \eqref{tocalculate} as
\begin{equation}
    \sum_{y_1,\dots,y_k\in\Fq} e_p(-\Tr\bigg((y_1,\dots,y_k,0,\dots,0)\cW \begin{pmatrix}
        \xi_1\\
        \vdots\\
        \xi_n
    \end{pmatrix}\bigg)).
\end{equation}
By Lemma \ref{lemfin}, we see that this sum $=q^k$, if $\cW \begin{pmatrix}
        \xi_1\\
        \vdots\\
        \xi_n
    \end{pmatrix}\in \{0\}^k\times \Fq^{n-k}$; and $=0$ otherwise. Therefore,
\[ V^*= \cW^{-1} (\{0\}^k\times \Fq^{n-k}) \]
is an $(n-k)$-dimensional subspace, and
\[ \wh \Id_V=q^k \cdot \Id_{V^*}. \]

To show $V^*=V^\perp$, we just need to check any vector $\cW^{-1}(0,\dots,0,a_1,\dots,a_{n-k})^T\in  \cW^{-1}(\{0\}^k\times \Fq^{n-k})$ is orthogonal to any $\bv_i$ ($1\le i\le k$). In other words,
\[ \bv_i \cW^{-1}\begin{pmatrix}
    0\\
    \vdots\\
    0\\
    a_1\\
    \vdots\\
    a_{n-k}
\end{pmatrix}=0. \]
This is true since $\bv_i \cW^{-1}=(0,\dots,0,1,0\dots,0)$ where the $i$-th entry is $1$.
\end{proof}

It is also not hard to see the following results, for which we omit the proof.
\begin{lemma}\label{lem2}
    If $V\in G(k,\Fq^n)$, then $(V^*)^*=V$. Therefore, $(\cdot)^*:G(k,\Fq^n)\rightarrow G(n-k,\Fq^n)$ is a bijection.
\end{lemma}

\begin{lemma}\label{lem3}
    For two subspaces $V,W$ in $\Fq^n$, we have $V\subset W\Leftrightarrow W^*\subset V^*$.
\end{lemma}

We also need a key lemma about the Falconer-type exceptional estimate.

\begin{lemma}\label{lemfalconer}
Let $A\subset \Fq^n$ be a set with $\#A=q^a$ $(0<a<n)$. For $s\in (0,a)$, recall
\begin{align*}
    E_s(A)&=\{ V\in G(k,\Fq^n):\#\pi_V(A)<q^s \}\\
    &=\{ V^*\in G(n-k,\Fq^n):\#\{\textup{translated~copies~of~}V^*\textup{~to~cover~}A\}<q^s \}. 
\end{align*} 
Let $M$ be the overlapping number of $\{V\setminus\{0\}: V\in E_s(A)\}$, i.e.,
\[M:=\sup_{\xi\in \Fq^n\setminus \{0\}}\sum_{V\in E_s(A)}\Id_{V}(\xi).\]
Then
\begin{equation}\label{falconereq}
    \#E_s(A)\lesssim M q^{n-k+s-a}. 
\end{equation}
\end{lemma}

\begin{remark}
{\rm
    Noting that for $\xi\in\Fq^n\setminus \{0\}$, we have
    \[ \sum_{V\in E_s(A)}\Id_{V}(\xi)\le \sum_{V\in G(k,\Fq^n)}\Id_{V}(\xi)=\# \{ V\in G(k,\Fq^n): 0,\xi\in V \}. \]
    Denote the line passing through $0,\xi$ by $\ell_\xi$. Noting that $\ell_\xi\subset V\Leftrightarrow V^*\subset \ell_\xi^*$ (by Lemma \ref{lem3}), we see that the right hand side of the inequality above is equal to
    \[\#\{W\in G(n-k,\Fq^n):W\subset \ell_\xi^*\}=\#G(n-k,\ell_\xi^*)\sim q^{(k-1)(n-k)}. \]
    We obtain that $M\lesssim q^{(k-1)(n-k)}$. Plugging into \eqref{falconereq}, we obtain that
    \begin{equation}\label{eqremark}
        \# E_s(A)\lesssim q^{k(n-k)+s-a},
    \end{equation}
    which is the Falconer-type estimate \ref{2}.
    }
\end{remark}

\begin{proof}[Proof of Lemma \ref{lemfalconer}]
By definition, for each $V \in E_s(A)$, there exists a set of $(n-k)$-planes $\mathcal L_V (=\pi_V(A))$ parallel to $V^*$ such that $A\subset\bigcup_{W \in \mathcal L_V} W$. Furthermore, we have  $\# \mathcal L_V < q^s$.

Let 
\[
f = \sum_{V \in E_s(A)} \sum_{W \in \mathcal L_V} \Id_W.
\]
We will apply the high-low method to $f$ using the Fourier transform on $\Fq^n$. Denote $\# E_s(A) = q^t$ for simplicity. Then, notice that for every $a\in A$ and for every $V \in E_s(A)$, there exists a $W\in\cL_V$ containing $a$. Therefore, we have that 
\begin{equation}\label{lefthandside}
    q^a q^{2t} = \#A(\# E_s(A))^2 \leq \int_A f^2 = \int_A \left(\sum_{V \in E_s(A)} \sum_{W \in \mathcal L_V} \Id_W \right)^2.
\end{equation}

We now seek to find an upperbound for the right hand side, for which we use the high-low method. The idea of the high-low method originates from \cite{vinh2011szemeredi}, \cite{guth2019incidence}, and has recently been applied to solve many problems. We briefly explain the idea of high-low method in the finite field setting. For a function $f$ on $\Fq^n$, we want to decompose it into high part and low part:
\[f=f_h+f_l.\]
The ``high part" $f_h$ satisfies $0\notin\supp \wh f_h$; the ``low part" $f_l$ satisfies $\supp\wh f_l\subset \{0\}$. By the requirement of the high part and low part, we can see that \[f_l(x)=\big(\frac{1}{q^n}\int_{\Fq^n}f(x)dx\big) \Id_{\Fq^n}(x),\ \ \ f_h(x)=f(x)-\big(\frac{1}{q^n}\int_{\Fq^n}f(x)dx\big) \Id_{\Fq^n}(x). \]
The Fourier support condition on $f_h$ will give us more orthogonality, and hence more gains when we use $L^2$ estimate.

We come back to the proof.
Notice that
\[
\int_A \left(\sum_{V \in E_s(A)} \sum_{W \in \mathcal L_V} \Id_W \right)^2 \lesssim \int_A \left(\sum_{V \in E_s(A)} \sum_{W \in \mathcal L_V} \Id_W - \frac{1}{q^k} \right)^2 + \int_A \left(\sum_{V \in E_s(A)} \sum_{W \in \mathcal L_V} \frac{1}{q^k} \right)^2.
\]
We now show that the first term on the right hand side dominates. To see this, notice that 
\[
\int_A \left(\sum_{V \in E_s(A)} \sum_{W \in \mathcal L_V} \frac{1}{q^{k}} \right)^2 = \int_A \left(\#E_s(A)\cdot  \#\mathcal L_V \cdot \frac{1}{q^{k}}\right)^2 \leq q^a\cdot q^{2(s+t-k)}.
\]
Notice that this is much less than the left hand side of \eqref{lefthandside} since $s<k$ and we may assume $q$ is large enough (since for small $q$, \eqref{mainthmest} naturally holds by choosing large enough constant $C_{n,k,a,s}$). Therefore, we have that 
\begin{align*}
    q^{a+2t} &\lesssim \int_A \left(\sum_{V \in E_s(A)} \sum_{W \in \mathcal L_V} \Id_W - \frac{1}{q^{k}} \right)^2 \\
    &\le \int_{\Fq^n} \left(\sum_{V \in E_s(A)} \sum_{W \in \mathcal L_V} \Id_W - \frac{1}{q^{k}} \right)^2.
\end{align*}
We now apply the Fourier transform to the last integrand.

Since any $W\in\cL_V$ is a translation of $V^*$, we can write $W=x_W+V^*$ for some $x_W\in\Fq^n$.
By Lemma \ref{lem1}, we have that \[\wh\Id_W(\xi)= 
e_p(-\Tr(x_W\cdot \xi))\wh 1_{V^*}= q^{n-k} e_p(-\Tr(x_W\cdot\xi))\Id_{V}.\] We also simply note that $\wh \Id_{\Fq^n}=q^n\Id_{\{0\}}$. We have 
\[
\left(\Id_W - \frac{1}{q^{k}}\right)^\wedge\,\,(\xi) = q^{n-k}e_p(-x_W\cdot\xi)\Id_{V}(\xi) - q^{n-k}\cdot \Id_{\{0\}}(\xi).
\]
Therefore, we see that $\supp \left(\sum_{W\in\cL_V}(\Id_W - \frac{1}{q^{k}})\right)^\wedge\subset V\setminus\{0\}$. Applying Plancherel and noting the definition of $M$, we have
\begin{align*}
q^{a+2t}&\lesssim \frac{1}{q^n}\int_{\Fq^n}\left|  \sum_{V\in E_{s}(A)}\sum_{W\in\cL_V}(\Id_W-\frac{1}{q^k})^\wedge\right|^2\\
&\lesssim \frac{1}{q^n}M\sum_{V\in E_{s}(A)}\int_{\Fq^n}\left|  \sum_{W\in\cL_V}(\Id_W-\frac{1}{q^k})^\wedge\right|^2\\
&=M\sum_{V\in E_{s}(A)}\int_{\Fq^n}\left|  \sum_{W\in\cL_V}\Id_W-\frac{1}{q^k}\right|^2\\
&\lesssim M\sum_{V\in E_{s}(A)}\int_{\Fq^n}  (\sum_{W\in\cL_V}\Id_W)^2+\int_{\Fq^n}(\#\cL_V)^2\frac{1}{q^{2k}}.
\end{align*}
Noting that $(\sum_{W\in\cL_V}\Id_W)^2=\sum_{W\in\cL_V}\Id_W$ (as $\{\Id_W\}_{W\in\cL_V}$ are disjoint), $\#\cL_V<q^s$, and $\#E_s(A)=q^t$, we see that the inequality above is 
\begin{align*}
    &\lesssim M q^t( q^{s+n-k}+q^{2s+n-2k} )\lesssim Mp^{t+s+n-k}.
\end{align*}
Combining with the lower bound $q^{a+2t}$, we obtain
\[ q^t\lesssim Mp^{n-k+s-a}. \]

\end{proof}

\section{Proof of Theorem \ref{mainthm}}\label{section5}

The goal of this section is to prove Theorem \ref{mainthm} which we restate here:
\begin{theorem}
    Let $A\subset \Fq^n$ be a set with $\#A=q^a$ $(0<a<n)$. For $s\in (0,a)$, define
    \[E_s(A):=\{ V\in G(k,\Fq^n): \#\pi_V(A)<q^s \}.  \]
Then for $s<\frac{a+2k-n}{2}$, we have
\begin{equation}\label{mainest}
    \# E_s(A)\le C_{n,k,a,s}\cdot\log q\cdot q^{t},
\end{equation}
where $t=\max\{k(n-k)+2(s-a), (k-1)(n-k)\}$. Here, $C_{n,k,a,s}$ is a constant that may depend on $n,k,a,s$, but not depend on $q$.
\end{theorem}

\subsection{Proof of Theorem \ref{mainthm}}

Let $A\subset \Fq^n$ with $\#A= q^a$. We consider two cases.

\medskip

\textbf{Case 1}: There exists a hyperplane, $H\in A(n-1,\Fq^n)$, such that 
\[
\#(A\cap H) \ge q^{s+n-k-1}.
\]
Let $H_0\in G(n-1,\Fq^n)$ be such that $H$ is parallel to $H_0$.
Then, we claim that every $V \in E_s(A)$ must satisfy $V^*\subset H_0$.

To see this, notice that if $V^*\in G(n-k,\Fq^n)$ is not contained in $H_0$, then $H\cap V^*$ is a $(n-k-1)$-dimensional plane, which means that $\#(H\cap V^*)= q^{n-k-1}$. Recalling the definition of $\pi_V$ in \eqref{defpiveq}, we have 
\[
\#\pi_V(A) \geq \#\pi_V(A\cap H) \ge \frac{\#(H\cap A)}{\#(H\cap V^*)}\geq q^s.
\]
So, $\{V^*:V\in E_s(A)\}\subset G(n-k,H_0)$.
It follows that 
\[\# E_s(A) \le \# G(n-k,H_0)\sim q^{(k-1)(n-k)}.\]

\bigskip

\textbf{Case 2}: Suppose for every hyperplane $H\in A(n-1,\Fq^n)$, we have that 
\[
\#(A\cap H) \leq q^{s+n-k-1}.
\]
First, we define
\[ M:=\sup_{\xi\in\Fq^n\setminus\{0\}}\sum_{V\in E_s(A)}\Id_{V}(\xi). \]
We denote $\#E_s(A)=q^t$.
By Lemma \ref{lemfalconer}, we have
\[
q^{t} \leq M q^{n-k+s-a}.
\]
To complete the proof, it remains to prove the following lemma.
\begin{lemma} \label{lemma:hyper1} 
\[
M  \lesssim \log q\cdot q^{(n-k)(k-1)+s-a}.
\]
\end{lemma}

\begin{proof}
Let $\xi_0$ be the point in $\Fq^n - \{0\}$ such that 
\begin{equation}\label{M}
   M = \#\{V\in E_s(A):\xi_0\in V\}.
\end{equation}

We know that such a $\xi_0$ exists since there are only finitely many $\xi$. 
Let $\ell$ be the line passing through $0$ and $\xi_0$, and let $H=\ell^*$ which is a hyperplane. We should view $\ell$ as a line in the frequency space, and view $H$ as a hyperplane in the physical space.
Define the set
\[
\Theta := \{V \in E_s(A):\ell\subset V\}=\{V\in E_s(A): V^*\subset H\}.
\]
By \eqref{M}, $\#\Theta=M$.

Now, we decompose $\Fq^n$ into $(n-1)$-dimensional planes that are parallel to $H$:
\[\Fq^n = \bigsqcup_{i=1}^q H_i.\]
Let $A_i = H_i \cap A$, which is the intersection of $A$ with each slice.
\begin{lemma}\label{lemma:hyper2}
  For each $i$,  we have that
\[
\sum_{V \in \Theta} \#\pi_V(A_i) \gtrsim \# \Theta \min\{q^{k-1}, \#A_i q^{-(n-k)(k-1)}\#\Theta\}.
\]
\end{lemma}

\begin{proof}
If $\#A_i p^{-(n-k)(k-1)}\#\Theta\le C$ for some large constant $C$, then the estimate trivially holds since $\#\pi_V(A_i)\ge 1$. Therefore, we assume $\#A_i p^{-(n-k)(k-1)}\#\Theta>C$.

    The proof is by applying \eqref{eqremark} to the set $A_i$ in the $(n-1)$-dimensional space $H_i$($\cong \Fq^{n-1}$). This is actually an exceptional set estimate for projection to $(k-1)$-planes in $\Fq^{n-1}$. 

    Let \begin{align*}
        E:&=\{ V^*\in G(n-k,H): \#\pi_V(A_i)<q^{s'} \}\\
        &=\{V^*\in G(n-k,H):\#\{\textup{translated~copies~of~}V^*\textup{~to~cover~}A_i\}<q^{s'} \},
    \end{align*} 
    where $s'$ is to be determined. \eqref{eqremark} yields that
    \[ \#E\lesssim q^{(k-1)(n-k)+s'}(\#A_i)^{-1}. \]
    We choose $s'$ to be such that
    \[ q^{s'}=C^{-1}\min\{ q^{k-1}, \#A_i q^{-(n-k)(k-1)}\#\Theta \}, \]
    where $C$ is a large constant. Plugging into the upper bound of $\#E$, we see that
    \[ \#E\le \frac{1}{2}\#\Theta. \]
    Therefore,
    \[ \sum_{V\in\Theta}\#\pi_V(A_i)\ge\sum_{V\in\Theta\setminus E}\#\pi_V(A_i)\ge \frac{1}{2}\#\Theta q^{s'}\gtrsim\# \Theta \min\{q^{k-1}, \#A_i q^{-(n-k)(k-1)}\#\Theta\}. \]
\end{proof}

We continue the proof. By a Fubini-type argument, we have
\begin{align*}
    \#\Theta \cdot q^s
    &\geq \sum_{V\in\Theta} \#\pi_V(A) \\
    &= \sum_{V\in\Theta} \sum_{i=1}^q \#\pi_V(A_i) \\
    &= \sum_{i=1}^q \sum_{V\in\Theta} \#\pi_V(A_i).
\intertext{Applying Lemma \ref{lemma:hyper2}, we have that}
   \#\Theta \cdot q^s  &\gtrsim \sum_{i=1}^q \#\Theta \min\{q^{k-1}, \# A_i q^{-(n-k)(k-1)}\#\Theta\}. 
\intertext{By dyadic pigeonholing, we choose $I$ which is a subset of these $i$, such that there exists $\beta>0$ with $\#A_i \sim q^\beta$ for $i\in I$, and $\#I \cdot q^\beta\gtrsim (\log q)^{-1}\#A$. Thus,}
\#\Theta \cdot q^s &\gtrsim \sum_{i\in I} \#\Theta \min\{q^{k-1}, q^{\beta -(n-k)(k-1)}\#\Theta \}.
\end{align*}
Also, recall the assumption at the beginning of \textbf{Case 2}: \begin{equation}\label{case2assumption}
    q^\beta\le q^{s+n-k-1}.
\end{equation}

We now have two cases depending on where the minimum is achieved. Firstly, if $q^{k-1} \geq q^{\beta - (n-k)(k-1)}\#\Theta$, we have that 
\[
\#\Theta\cdot q^s \gtrsim \#I \cdot q^\beta \cdot q^{-(n-k)(k-1)}(\#\Theta)^2 \gtrsim (\log q)^{-1} q^{a-(n-k)(k-1)}(\#\Theta)^2.
\]
Therefore, $M=\#\Theta\leq \log q \cdot q^{(n-k)(k-1)+s-a}$, which finishes the proof of Lemma \ref{lemma:hyper1}.

The second scenario is $q^{k-1} \leq q^{\beta - (n-k)(k-1)}\#\Theta$. We will show that this will not happen. If it happens, we have 
\[
\#\Theta\cdot q^s \gtrsim \#I\cdot \#\Theta\cdot q^{k-1}.
\]
Multiplying $q^\beta$ on both sides gives
\[q^\beta \#I \cdot q^{k-1} \lesssim q^\beta q^s.\]
This together with \eqref{case2assumption} implies that 
\[
q^{a+(k-1)}(\log q)^{-1} \lesssim q^{s+\beta} \le q^{2s+n-k-1}.
\]
When $q$ is big enough,
this is a contradiction, as we assumed that $s<\frac{a+2k-n}{2}$. Thus, Lemma \ref{lemma:hyper1} is proved.

\end{proof}

\begin{remark}
 {\rm 
    The main obstacle to generalize the proof to $\R^n$ is as follows. Let $A\subset \R^n$ with $\dim(A)=a$. Let $\{V_t\}_{t\in \R}$ be the one-parameter family of $(n-1)$-planes, where each $V_t$ is orthogonal to $(0,\dots,0,1)$ and intersects with the $x_n$-axis at $(0,\dots,0,t)$. Set $A_t=A\cap V_t$. If everything is finite, then $\#A=\sum_t \#A_t$,
which implies there exists $1\le M\le \#A$ such that
\[ \#A \lesssim \log(\#A)\cdot M\cdot\#\{t: \#A_t\le M\}.\]
In the continuous setting, we hope there exists $\beta\in(0,a]$ such that
\[ \dim(A)\le \beta+\dim(\{ t: \dim(A_t)\ge\beta \}).  \]
This roughly says that if $A$ is big, then we can find many big slices $\{A_t\}$ of $A$. If we replace ``$\le$" by ``$\ge$" in the inequality above, then it is always true. However, it may fail in the reverse direction. Actually, there exists a set $A$ with $\dim(A)=n$ but $\dim(A_t)=0$ for all $t$. This failure of Fubini-type argument is the main obstacle to generalize our theorem to $\R^n$.}
\end{remark}

\bibliographystyle{abbrv}
\bibliography{bibli}

\end{document}